\documentclass[twoside,leqno,10pt, A4]{amsart}
\usepackage{amsfonts}
\usepackage{amsmath}
\usepackage{amscd}
\usepackage{amssymb}
\usepackage{amsthm}
\usepackage{amsrefs}
\usepackage{latexsym}
\usepackage{mathrsfs}
\usepackage{bbm}
\usepackage{enumerate}
\usepackage{graphicx}

\usepackage{amsfonts}
\usepackage{amsmath}
\usepackage{amscd}
\usepackage{amssymb}
\usepackage{amsthm}
\usepackage{amsrefs}
\usepackage{latexsym}
\usepackage{mathrsfs}
\usepackage{bbm}
\usepackage{amscd}
\usepackage{amssymb}
\usepackage{amsthm}
\usepackage{amsrefs}
\usepackage{latexsym}
\usepackage{mathrsfs}
\usepackage{bbm}
\usepackage{enumerate}
\usepackage{graphicx}
\usepackage{color}
\begin{document}

\newtheorem{theorem}[subsection]{Theorem}
\newtheorem{proposition}[subsection]{Proposition}
\newtheorem{lemma}[subsection]{Lemma}
\newtheorem{corollary}[subsection]{Corollary}
\newtheorem{conjecture}[subsection]{Conjecture}
\newtheorem{prop}[subsection]{Proposition}
\numberwithin{equation}{section}
\newcommand{\mr}{\ensuremath{\mathbb R}}
\newcommand{\mc}{\ensuremath{\mathbb C}}
\newcommand{\dif}{\mathrm{d}}
\newcommand{\intz}{\mathbb{Z}}
\newcommand{\ratq}{\mathbb{Q}}
\newcommand{\natn}{\mathbb{N}}
\newcommand{\comc}{\mathbb{C}}
\newcommand{\rear}{\mathbb{R}}
\newcommand{\prip}{\mathbb{P}}
\newcommand{\uph}{\mathbb{H}}
\newcommand{\fief}{\mathbb{F}}
\newcommand{\majorarc}{\mathfrak{M}}
\newcommand{\minorarc}{\mathfrak{m}}
\newcommand{\sings}{\mathfrak{S}}
\newcommand{\fA}{\ensuremath{\mathfrak A}}
\newcommand{\mn}{\ensuremath{\mathbb N}}
\newcommand{\mq}{\ensuremath{\mathbb Q}}
\newcommand{\half}{\tfrac{1}{2}}
\newcommand{\f}{f\times \chi}
\newcommand{\summ}{\mathop{{\sum}^{\star}}}
\newcommand{\chiq}{\chi \bmod q}
\newcommand{\chidb}{\chi \bmod db}
\newcommand{\chid}{\chi \bmod d}
\newcommand{\sym}{\text{sym}^2}
\newcommand{\hhalf}{\tfrac{1}{2}}
\newcommand{\sumstar}{\sideset{}{^*}\sum}
\newcommand{\sumprime}{\sideset{}{'}\sum}
\newcommand{\sumprimeprime}{\sideset{}{''}\sum}
\newcommand{\sumflat}{\sideset{}{^\flat}\sum}
\newcommand{\shortmod}{\ensuremath{\negthickspace \negthickspace \negthickspace \pmod}}
\newcommand{\V}{V\left(\frac{nm}{q^2}\right)}
\newcommand{\sumi}{\mathop{{\sum}^{\dagger}}}
\newcommand{\mz}{\ensuremath{\mathbb Z}}
\newcommand{\leg}[2]{\left(\frac{#1}{#2}\right)}
\newcommand{\muK}{\mu_{\omega}}
\newcommand{\thalf}{\tfrac12}
\newcommand{\lp}{\left(}
\newcommand{\rp}{\right)}
\newcommand{\Lam}{\Lambda_{[i]}}
\newcommand{\lam}{\lambda}
\def\L{\fracwithdelims}
\def\om{\omega}
\def\pbar{\overline{\psi}}
\def\phi{\varphi}
\def\lam{\lambda}
\def\lbar{\overline{\lambda}}
\newcommand\Sum{\Cal S}
\def\Lam{\Lambda}
\newcommand{\sumtt}{\underset{(d,2)=1}{{\sum}^*}}
\newcommand{\sumt}{\underset{(d,2)=1}{\sum \nolimits^{*}} \widetilde w\left( \frac dX \right) }

\theoremstyle{plain}
\newtheorem{conj}{Conjecture}
\newtheorem{remark}[subsection]{Remark}

\providecommand{\re}{\mathop{\rm Re}}
\providecommand{\im}{\mathop{\rm Im}}
\def\cI{\mathcal{I}}
\def\cL{\mathcal{L}}
\def\E{\mathbb{E}}

\makeatletter
\def\widebreve{\mathpalette\wide@breve}
\def\wide@breve#1#2{\sbox\z@{$#1#2$}%
     \mathop{\vbox{\m@th\ialign{##\crcr
\kern0.08em\brevefill#1{0.8\wd\z@}\crcr\noalign{\nointerlineskip}%
                    $\hss#1#2\hss$\crcr}}}\limits}
\def\brevefill#1#2{$\m@th\sbox\tw@{$#1($}%
  \hss\resizebox{#2}{\wd\tw@}{\rotatebox[origin=c]{90}{\upshape(}}\hss$}
\makeatletter

\title[Sharp lower bounds for moments of $\zeta'(\rho)$]{Sharp lower bounds for moments of $\zeta'(\rho)$}

\author{Peng Gao}
\address{School of Mathematical Sciences, Beihang University, Beijing 100191, P. R. China}
\email{penggao@buaa.edu.cn}
\begin{abstract}
 We study the $2k$-th discrete moment of the derivative of the Riemann zeta-function at nontrivial zeros to establish sharp lower bounds for all real $k \geq 0$ under the Riemann hypothesis (RH).
\end{abstract}

\maketitle



\section{Introduction}
\label{sec 1}

  Various types of moments of the Riemann zeta function $\zeta(s)$ have been extensively studied in the literature. In this paper, we are interested in the $2k$-th discrete moment of the derivative of $\zeta(s)$ at nontrivial zeros denoted by
\begin{align*}
J_k(T) :=\frac{1}{N(T)}\sum_{0<\Im(\rho)\le T}|\zeta'(\rho)|^{2k},
\end{align*}
  where we write $\rho$ for the nontrivial zeros of $\zeta(s)$ and
\begin{align*}
 N(T) =\sum_{0<\Im(\rho)\le T}1.
\end{align*}

   In \cite{Gonek}, S. M. Gonek initiated the study on $J_k(T)$ to show that under the truth of the Riemann hypothesis (RH), one has asymptotically
\begin{align*}
J_1(T) \sim \frac 1{12}(\log T)^3.
\end{align*}

  Regarding the order of magnitude for $J_k(T)$, S. M. Gonek \cite{Gonek1} and D. Hejhal \cite{Hejhal} independently conjectured that for any real $k$,
\begin{align}
\label{Jksim}
J_k(T) \asymp (\log T)^{k(k+2)}.
\end{align}

   A precisely asymptotic formula is further conjectured by C. P. Hughes, J. P. Keating and N. O'Connell in \cite{HKO}, building on connections with the random matrix theory. The evidence from the random matrix side also suggests that \eqref{Jksim} may not be valid for $k \leq -3/2$.

  A proof of \eqref{Jksim} for the case of $k=2$ is given by N. Ng \cite{Ng} assuming RH. Under RH and the additional assumption that
the zeros of $\zeta(s)$ are simple, S. M. Gonek \cite{Gonek1} obtained a lower bound for $J_{-1}(T)$ of the conjectured order of magnitude. An explicit estimation for the constant involved is further given by M. B. Milinovich and N. Ng \cite{MN}. In \cite{MN1}, M. B. Milinovich and N. Ng also proved that for all natural number $k$,
\begin{align}
\label{Jlowerbound}
J_k(T) \gg_k (\log T)^{k(k+2)}.
\end{align}

  On the other hand, M. B. Milinovich \cite{Milinovich} proved that under RH, for any natural number $k$ and any $\varepsilon>0$,
\begin{align}
\label{Jupperbound}
J_k(T) \ll_{k, \varepsilon} (\log T)^{k(k+2)+\varepsilon}.
\end{align}

  In \cite{Kirila}, S. Kirila obtained sharp upper bounds for $J_k(T)$ for all real $k>0$ under RH. In particular, this implies the validity of \eqref{Jupperbound} for natural numbers $k$ without the extra $\varepsilon$ power. Together with \eqref{Jlowerbound}, we see that \eqref{Jksim} is valid for all natural numbers $k$.

 The approach taken by M. B. Milinovich and N. Ng to obtain their result concerning \eqref{Jlowerbound} follows from a simple and powerful method developed by Z. Rudnick and K. Soundararajan in \cites{R&Sound, R&Sound1} for establishing lower bounds for moments of families of $L$-functions, while the approach taken by S. Kirila follows from a method of K. Soundararajan \cite{Sound01} together with its refinement by A. J. Harper \cite{Harper} for establishing upper bounds for moments of families of $L$-functions.

   There are now a few more approaches towards establishing sharp bounds for moments of $L$-functions, notably the upper bounds principal of M. Radziwi{\l\l} and K. Soundararajan \cite{Radziwill&Sound} and the lower bounds principal of W. Heap and K. Soundararajan \cite{H&Sound}. One then expects to apply these approaches to obtain sharp bounds concerning $J_k(T)$. In fact, it is pointed out in \cite{Kirila} that one should be able to establish sharp lower bounds for all real $k>0$ using the approaches in \cite{Radziwill&Sound, Radziwill&Sound1}. The aim of this paper is to achieve this and our main result is as follows.
\begin{theorem}
\label{thmlowerboundJ}
   Assuming RH. For large $T$ and any $k \geq 0$, we have
\begin{align*}
   J_k(T) \gg_k (\log T)^{k(k+2)}.
\end{align*}
\end{theorem}

   We shall instead apply the lower bounds principal of W. Heap and K. Soundararajan \cite{H&Sound} in the proof of Theorem \ref{thmlowerboundJ}. The proof also uses the arguments by A. J. Harper in  \cite{Harper} and by S. Kirila in \cite{Kirila}. Combining Theorem \ref{thmlowerboundJ} and the above mentioned result of S. Kirila in \cite{Kirila}, we immediately obtain the following result concerning the order
of magnitude of $J_k(T)$.
\begin{corollary}
\label{cororderofmag}
   Assuming RH. For large $T$, the estimation given in  \eqref{Jksim} is valid for all real $k \geq 0$.
\end{corollary}

\section{Preliminaries}
\label{sec 2}

 We now introduce some notations and auxiliary results used in the paper. We assume the truth of RH throughout so that we may write each nontrivial zero $\rho$ of $\zeta(s)$ as $\rho=\half+i\gamma$, where we denote $\gamma \in \mr$ for the imaginary part of $\rho$. We denote  $\omega(n)$ for the number of distinct prime factors of $n$ and $\Omega(n)$ for the number of prime powers dividing $n$.  We note the following estimation (see \cite[Theorem 2.10]{MVa}) for $\omega(n)$,
\begin{align}
\label{omegabound}
  \omega(n) \leq \frac {\log n}{\log \log n}(1+O(\frac {1}{\log \log n})), \quad n \geq 3.
\end{align}

   We also define $\Lambda_j(n)$ for all integers $j \geq 0$ to be the coefficient of $n^{-s}$ in the Dirichlet series expansion of $(-1)^{j}\zeta^{(j)}(s)/\zeta(s)$. In particular, we have $\Lambda_1(n)=\Lambda(n)$, the usual von Mangoldt function.  We extend the definition of $\Lambda$ to all real numbers $x$ by defining $\Lambda(x)=0$ when $x$ is not an integer and we note the following uniform version of Landau’s
formula \cite{Landau1912}, originally proved by  S. M. Gonek \cite{Gonek93}.
\begin{lemma}
\label{Lem-Landau}
	Assume RH. Then we have for $T$ large and any positive integers $a, b$,
\begin{align}
\label{sumgamma}
\begin{split}
	\sum_{T<\gamma\le 2T}(a/b)^{i\gamma}=
\begin{cases}
N(T, 2T), \quad a=b, \\
\displaystyle -\frac{T}{2\pi}\frac{\Lambda(a/b)}{\sqrt{a/b}}+O\big(\sqrt{ab}(\log T)^2\big), \quad a>b, \\
\displaystyle -\frac{T}{2\pi}\frac{\Lambda(b/a)}{\sqrt{b/a}}+O\big(\sqrt{ab}(\log T)^2\big), \quad b>a,
\end{cases}
\end{split}
\end{align}
  where we denote $N(T, 2T)$ for the number of nontrivial zeros $\rho$ such that $T < \Im(\rho) \leq 2T$.
\end{lemma}

  The cases when $a \neq b$ of Lemma \ref{Lem-Landau} are given in \cite[Lemma 5.1]{Kirila} while the case $a =b$ of Lemma \ref{Lem-Landau} is trivial. Recall that the Riemann–von Mangoldt formula asserts (see \cite[Chapter 15]{Da}) that
\begin{align*}
   N(T)=\frac {T}{2\pi}\log \frac {T}{2\pi e}+O(\log T).
\end{align*}
 It follows from this and the relation $N(T, 2T)=N(2T)-N(T)$ that
\begin{align}
\label{N2Tbound}
   N(T, 2T) \ll T \log T.
\end{align}

   We reserve the letter $p$ for a prime number in this paper and we note that $\Lambda_2(n)$ is supported on integers $n$ with $\omega(n) \leq 2$ satisfying for primes $p, q, p \neq q$ and positive integers $i, j$,
\begin{align}
\label{Lambda2}
   \Lambda_2(p^i) \ll i(\log p)^2, \quad \Lambda_2(p^iq^j) \ll (\log p)(\log q).
\end{align}

  We recall the following well-known Mertens' formula (see \cite[Theorem 2.7]{MVa1}).
\begin{lemma} \label{RS} Let $x \geq 2$. We have, for some constant $b$,
$$
\sum_{p\le x} \frac{1}{p} = \log \log x + b+ O\Big(\frac{1}{\log x}\Big).
$$
\end{lemma}

  We end this section by including a mean value theorem given in \cite[Lemma 4.1]{MN1} concerning integrals over Dirichlet polynomials.
\begin{lemma}
\label{Lem-MVDP}
	Let $\{a_n\}$ and $\{b_n\}$ be sequences of complex numbers. Let $T_1$ and $T_2$ be positive real numbers and $g(t)$ be a real-valued function that is continuously differentiable on the interval $[T_1, T_2]$. Then
\begin{align*}
\begin{split}
& \int^{T_2}_{T_1}g(t)\left ( \sum^{\infty}_{n=1}a_nn^{-it}\right )\left( \sum^{\infty}_{n=1}b_nn^{it}\right )dt \\
=& \int^{T_2}_{T_1}g(t)dt\sum^{\infty}_{n=1}a_nb_n+O\left ( \left (|g(T_1)|+|g(T_2)|+\int^{T_2}_{T_1}|g'(t)|dt \right )\left ( \sum^{\infty}_{n=1}n|a_n|^2\right )^{1/2}\left( \sum^{\infty}_{n=1}n|b_n|^2 \right )^{1/2}\right ) .
\end{split}
\end{align*}
\end{lemma}

\section{Proof of Theorem \ref{thmlowerboundJ}}
\label{sec 2'}

\subsection{The lower bound principle}

    We assume that $T$ is a large number throughout the proof. As the case $k=0$ is trivial, we only consider the case $k>0$ in the proof.   
Moreover, we note that in the rest of the paper, the explicit constants involved in estimations using $\ll$ or the big-$O$ notations depend on $k$ only and are uniform with
 respect to $\rho$. We further make the convention that an empty product is defined to be $1$.

 We follow the ideas of A. J. Harper in \cite{Harper} and the notations of S. Kirila in \cite{Kirila} to define for a large number $M$ depending on $k$ only,
$$ \alpha_{0} = 0, \;\;\;\;\; \alpha_{j} = \frac{20^{j-1}}{(\log\log T)^{2}} \;\;\; \forall \; j \geq 1, \quad
\mathcal{J} = \mathcal{I}_{k,T} = 1 + \max\{j : \alpha_{j} \leq 10^{-M} \} . $$

   It follows from the above notations and Lemma \ref{RS} that we have for $1 \leq j \leq \mathcal{J}-1$ and $T$ large enough,
\begin{align}
\label{sumpj}
 \sum_{T^{\alpha_{j}} < p \leq T^{\alpha_{j+1}}} \frac{1}{p}
 = \log \alpha_{j+1} - \log \alpha_{j} + o(1) = \log 20 + o(1) \leq 10.
\end{align}

  We denote for any real number $\ell$ and any $x \in \mr$,
\begin{align*}
  E_{\ell}(x) = \sum_{j=0}^{\lceil \ell \rceil} \frac {x^{j}}{j!}.
\end{align*}

 We then define for any real number $\alpha$ and any $1\leq j \leq \mathcal{J}$,
\begin{align*}
 {\mathcal P}_j(s)=&  \sum_{ p \in I_j}  \frac{1}{p^s}, \quad {\mathcal N}_j(s, \alpha) = E_{e^2k\alpha^{-3/4}_j} \Big (\alpha {\mathcal P}_j(s) \Big ), \quad  {\mathcal N}(s, \alpha)=  \prod^{\mathcal{J}}_{j=1} {\mathcal N}_j(s, \alpha).
\end{align*}

  Denote $g(n)$ for the multiplicative function given on prime powers by $g(p^{r}) = 1/r!$ and define functions $b_j(n), 1 \leq j \leq {\mathcal{J}}$  such that $b_j(n)=0$ or $1$ and that $b_j(n)=1$ only when $n$ is composed of at most $\lceil e^2k\alpha^{-3/4}_j \rceil$ primes, all from the interval $I_j$. We then have
\begin{align*}
 {\mathcal N}_j(s,\alpha) = \sum_{n_j}  \frac{\alpha^{\Omega(n_j)}}{g(n_j)}  b_j(n_j) \frac 1{n^s_j}, \quad 1\le j \le {\mathcal{J}}.
\end{align*}

  Note that each ${\mathcal N}_j(s,\alpha)$ is a short Dirichlet polynomial of length at most $T^{\alpha_{j}\lceil e^2k\alpha^{-3/4}_j \rceil}$. By taking $T$ large enough, we notice that
\begin{align*}
 \sum^{\mathcal{J}}_{j=1} \alpha_{j}\lceil e^2k\alpha^{-3/4}_j \rceil \leq 40e^2k10^{-M/4}.
\end{align*}
   It follows that ${\mathcal N}(s, \alpha)$ is also a short Dirichlet polynomial of length at most $T^{40 e^2k10^{-M/4}}$.

   Moreover, we write for simplicity that
\begin{align}
\label{Nskexpression}
 {\mathcal N}(s, \alpha)= \sum_{n} \frac{a_{\alpha}(n)}{n^s}.
\end{align}
   We note that $a_{\alpha}(n) \neq 0$ only when $n=\prod_{1\leq j \leq \mathcal{J}}n_j$, in which case we have
\begin{align*}
  a_{\alpha}(n)= \prod_{n_j}  \frac{\alpha^{\Omega(n_j)}}{g(n_j)}  b_j(n_j).
\end{align*}

   Taking note also the estimation that for all integers $i \geq 0$,
\begin{align*}
  \frac {\alpha^{i}}{i!} \ll  e^{|\alpha|},
\end{align*}
   we conclude from above discussions that for all $n$,
\begin{align}
\label{anbound}
  a_{\alpha}(n) \leq e^{|\alpha|\omega(n)}, \quad a_k(n) = 0, \text{ when } n > T^{40 e^2k10^{-M/4}}.
\end{align}

   We apply the above estimations and \eqref{omegabound} in \eqref{Nskexpression} to see that for $\Re(s) \geq -1/\log T$ and $T$ large enough,
\begin{align}
\label{Nnormbound}
 |{\mathcal N}(s,\alpha)| \ll e^{|\alpha|\frac {\log T}{\log \log T}(1+O(\frac {1}{\log \log T}))}T^{40e^2k10^{-M/4}(1+1/\log T)}.
\end{align}

  In the proof of Theorem \ref{thmlowerboundJ}, we need the following bounds concerning expressions involving with various  ${\mathcal N}(s, \alpha)$ .
\begin{lemma}
\label{lemNbounds}
 With the notations as above, we have for $0<k \leq 1/2$ and $1 \leq j \leq \mathcal{J}$,
\begin{align}
\label{est0}
\begin{split}
|\mathcal{N}_j(s,
 k)|^{2/k}|\mathcal{N}_j(s, k-1)|^{2} \le
|{\mathcal N}_j(s, k)|^2 \left( 1+e^{-e^2k\alpha^{-3/4}_j} \right )^{2/k+2}\left( 1-e^{-e^2k\alpha^{-3/4}_j} \right )^{-2} + |{\mathcal Q}_j(s, k)|^{2r_k}.
\end{split}
\end{align}

  We also have  for $k >1/2$ and $1 \leq j \leq \mathcal{J}$,
\begin{align}
\label{est0'}
\begin{split}
|\mathcal{N}_j(s, k-1)\mathcal{N}_j(s, k)|^{\frac {2k}{2k-1}}   \le
|{\mathcal N}_j(s, k)|^2  \left( 1+ e^{-e^2k\alpha^{-3/4}_j}
\right)^{\frac {2k}{2k-1}}\left( 1-e^{-e^2k\alpha^{-3/4}_j} \right )^{-2} + |{\mathcal Q}_j(s, k)|^{2r_k}.
\end{split}
\end{align}
  Here the implied constants in \eqref{est0} and \eqref{est0'} are absolute, and we define
$$
{\mathcal Q}_j(s,k) =\Big( \frac{64 \max (2, k+3/2 ) {\mathcal P}_j(s)}{\lceil e^2k\alpha^{-3/4}_j \rceil} \Big)^{ \lceil e^2k\alpha^{-3/4}_j \rceil},
$$
  with $r_k=2+\lceil 1/k \rceil$ for $0<k \leq 1/2$ and $r_k=1+\lceil 2k/(2k-1) \rceil$ for $k >1/2$.
\end{lemma}
\begin{proof}
  As in the proof of \cite[Lemma 3.4]{Gao2021-3}, we have for $|z| \le aK/20$ with $0<a \leq 2$,
\begin{align}
\label{Ebound}
\Big| \sum_{r=0}^K \frac{z^r}{r!} - e^z \Big| \le \frac{|z|^{K}}{K!} \le \Big(\frac{a e}{20}\Big)^{K}.
\end{align}
  By taking $z=\alpha {\mathcal P}_j(s), K=\lceil e^2k\alpha^{-3/4}_j \rceil$ and $a=\min (|\alpha|, 2 )$ in \eqref{Ebound}, we see that when $|{\mathcal P}_j(s)| \le \lceil e^2k\alpha^{-3/4}_j \rceil/(20(1+|\alpha|))$,
\begin{align*}
{\mathcal N}_j(s, \alpha) \leq & \exp ( \alpha{\mathcal P}_j(s) )\left( 1+ \exp ( |\alpha {\mathcal P}_j(s)| ) \left( \frac{a e}{20} \right)^{ e^2k\alpha^{-3/4}_j} \right) \leq \exp ( \alpha {\mathcal P}_j(s)  ) \left( 1+   e^{-e^2k\alpha^{-3/4}_j}  \right).
\end{align*}

  Similarly, we have
\begin{align*}
{\mathcal N}_(s, \alpha) \geq & \exp ( \alpha{\mathcal P}_j(s) )\left( 1- \exp ( |\alpha {\mathcal P}_j(s)| ) \left( \frac{a e}{20} \right)^{ e^2k\alpha^{-3/4}_j} \right) \geq \exp ( \alpha {\mathcal P}_j(s)  ) \left( 1-   e^{-e^2k\alpha^{-3/4}_j}  \right).
\end{align*}

  We apply the above estimation to ${\mathcal N}_j(s, k), {\mathcal N}_j(s, k-1)$ to see that when $0<k \leq 1/2$ and  $|{\mathcal P}_j(s)| \le \lceil e^2k\alpha^{-3/4}_j \rceil/60$, we have
\begin{align}
\label{est1}
\begin{split}
|{\mathcal N}_j(s, k)|^{\frac {2}{k}} |{\mathcal N}_j(s, k-1)|^{2}
\leq & \exp ( 2k \Re {\mathcal P}_j(s)  ) \left( 1+  e^{-e^2k\alpha^{-3/4}_j}  \right)^{2/k+2} \\
 \leq & |{\mathcal N}_j(s, k)|^2 \left( 1+e^{-e^2k\alpha^{-3/4}_j} \right )^{2/k+2}\left( 1-e^{-e^2k\alpha^{-3/4}_j} \right )^{-2}.
\end{split}
\end{align}

  The above arguments also imply that when $k>1/2$ and $|{\mathcal P}_j(s)| \le \lceil e^2k\alpha^{-3/4}_j \rceil/60$, then
\begin{align}
\label{est1'}
\begin{split}
|\mathcal{N}_j(s,
 2k-1)|^{\frac {2k}{2k-1}}  \le
|{\mathcal N}_j(s, k)|^2  \left( 1+ e^{-e^2k\alpha^{-3/4}_j}
\right)^{\frac {2k}{2k-1}}\left( 1-e^{-e^2k\alpha^{-3/4}_j} \right )^{-2} .
\end{split}
\end{align}

 On the other hand, when $|{\mathcal P}_j(s)| \ge \lceil e^2k\alpha^{-3/4}_j \rceil/60$, we have that
\begin{align}
\label{4.2}
\begin{split}
{\mathcal N}_j(s, \alpha) \le \sum_{r=0}^{\lceil e^2k\alpha^{-3/4}_j \rceil} \frac{|\alpha{\mathcal P}_j(s)|^r}{r!} & \le
|(|\alpha|+1) {\mathcal P}_j(s)|^{\lceil e^2k\alpha^{-3/4}_j \rceil} \sum_{r=0}^{\lceil e^2k\alpha^{-3/4}_j \rceil} \Big( \frac{60}{\lceil e^2k\alpha^{-3/4}_j \rceil}\Big)^{\lceil e^2k\alpha^{-3/4}_j \rceil-r} \frac{1}{r!}  \\
&   \le \Big( \frac{64(|\alpha|+1) |{\mathcal
P}_j(s)|}{\lceil e^2k\alpha^{-3/4}_j \rceil}\Big)^{\lceil e^2k\alpha^{-3/4}_j \rceil} .
\end{split}
\end{align}

   We then set $\alpha=1$ in the last expression in \eqref{4.2} to deduce that when $0<k \leq 1/2$ and $|{\mathcal P}_j(s)| \ge \lceil e^2k\alpha^{-3/4}_j \rceil/60$,
\begin{align}
\label{est2}
\begin{split}
|{\mathcal N}_j(s, k)|^{\frac {2}{k}} |{\mathcal N}_j(s, k-1)|^{2} \leq
|{\mathcal Q}_j(s,k)|^{2r_k}.
\end{split}
\end{align}

   Moreover, we set $\alpha=k$ in the last expression in \eqref{4.2} to deduce that when $k>1/2$ and $|{\mathcal P}_j(s)| \ge \lceil e^2k\alpha^{-3/4}_j \rceil/(10(1+|\alpha|))$,
\begin{align}
\label{est2'}
\begin{split}
 |{\mathcal N}_j(s, k-1){\mathcal N}_j(s, k)|^{\frac {2k}{2k-1}}  \le |{\mathcal Q}_j(s,k)|^{2r_k}.
\end{split}
\end{align}

   The assertion of the lemma now follows from \eqref{est1}, \eqref{est1'}, \eqref{est2} and \eqref{est2'}.
\end{proof}

  Next, we state the needed lower bounds principle of W. Heap and K. Soundararajan in
  \cite{H&Sound} for our situation.
\begin{lemma}
\label{lem1}
 With notations as above. For $0<k \leq 1/2$, we have
\begin{align}
\label{basiclowerbound}
\begin{split}
\sum_{0<\gamma\leq T}\zeta'(\rho)\mathcal{N}(\rho, k-1)\mathcal{N}(\overline{\rho}, k)
 \ll & \Big ( \sum_{0<\gamma\leq T}|\zeta'(\rho)|^{2k} \Big )^{1/2}\Big ( \sum_{0<\gamma\leq T}|\zeta'(\rho)|^{2} |\mathcal{N}(\rho, k-1)|^2  \Big)^{(1-k)/2} \\
 & \times \Big ( \sum_{0<\gamma\leq T}  \prod^{\mathcal{J}}_{i=1}\big ( |{\mathcal N}_i(\rho, k)|^2+ |{\mathcal Q}_i(\rho,k)|^{2r_k} \big )
 \Big)^{k/2}.
\end{split}
\end{align}
 For $k>1/2$, we have
\begin{align}
\label{basicboundkbig}
\begin{split}
 & \sum_{0<\gamma\leq T}\zeta'(\rho)\mathcal{N}(\rho, k-1)\mathcal{N}(\overline{\rho}, k)
 \leq  \Big ( \sum_{0<\gamma\leq T}|\zeta'(\rho)|^{2k} \Big )^{\frac {1}{2k}}\Big (  \sum_{0<\gamma\leq T}\prod^{\mathcal{J}}_{j=1} \big ( |{\mathcal N}_j(\rho, k)|^2+ |{\mathcal Q}_j(\rho,k)|^{2r_k} \big ) \Big)^{\frac {2k-1}{2k}}.
\end{split}
\end{align}
  The implied constants in \eqref{basiclowerbound} and \eqref{basicboundkbig} depend on $k$ only.
\end{lemma}
\begin{proof}
   We assume $0<k \leq 1/2$ first and apply H\"older's inequality to see that the left side of \eqref{basiclowerbound} is
\begin{align}
\label{basicbound0}
\begin{split}
 \leq & \Big ( \sum_{0<\gamma\leq T}|\zeta'(\rho)|^{2k} \Big )^{1/2}\Big ( \sum_{0<\gamma\leq T}|\zeta'(\rho)|^{2} |\mathcal{N}(\rho, k-1)|^2   \Big)^{(1-k)/2}\Big (\sum_{0<\gamma\leq T} |\mathcal{N}(\rho,
 k)|^{2/k}|\mathcal{N}(\rho, k-1)|^{2}  \Big)^{k/2}.
\end{split}
\end{align}

  We apply the estimation in \eqref{est0} in the last sum of \eqref{basicbound0} above and note that, upon applying the estimation that $1-e^{-u} \geq e^{-1/u}$ for $u>0$, we have
\begin{align*}
 \prod_{1 \leq j \leq {\mathcal J}} \left( 1+e^{-e^2k\alpha^{-3/4}_j} \right )^{2/k+2}\left( 1-e^{-e^2k\alpha^{-3/4}_j} \right )^{-2}
\leq & \prod_{1 \leq j \leq {\mathcal J}} \left( 1-e^{-e^2k\alpha^{-3/4}_j} \right )^{-2}
\leq \prod_{1 \leq j \leq {\mathcal J}}e^{2\alpha^{3/4}_j/(e^2k)}<\infty.
\end{align*}
  This leads to the estimation given in \eqref{basiclowerbound}.

  It remains to consider the case $k>1/2$ and we apply H\"older's inequality again to see that the left side of \eqref{basicboundkbig} is
\begin{align}
\label{holderkbig}
\begin{split}
 \leq  \Big ( \sum_{0<\gamma\leq T}|\zeta'(\rho)|^{2k} \Big )^{\frac {1}{2k}}\Big ( \sum_{0<\gamma\leq T}|\mathcal{N}(\rho, k-1)\mathcal{N}(\rho, k)|^{\frac {2k}{2k-1}}  \Big)^{\frac {2k-1}{2k}}.
\end{split}
\end{align}

  We apply the estimation in \eqref{est0'} in the last sum of  \eqref{holderkbig} above and note that the product
$$\prod_{1 \leq j \leq {\mathcal J}} \left( 1+ e^{-e^2k\alpha^{-3/4}_j}
\right)^{\frac {2k}{2k-1}}\left( 1-e^{-e^2k\alpha^{-3/4}_j} \right )^{-2} <\infty. $$
  This leads to the estimation given in \eqref{basicboundkbig} and this completes the proof.
\end{proof}

  It follows from the above lemma and the observation that $\overline{\rho}=1-\rho$ that in order to establish Theorem \ref{thmlowerboundJ}, it suffices to prove the following three propositions.
\begin{proposition}
\label{Prop4}
  With notations as above, we have for $k > 0$,
\begin{align}
\label{L1estmation}
 \sum_{0<\gamma\leq T}\zeta'(\rho)\mathcal{N}(\rho, k-1)\mathcal{N}(1-\rho, k) \gg T (\log T)^{ k^2+2}.
\end{align}
\end{proposition}

\begin{proposition}
\label{Prop5}
  With notations as above, we have for $k >0$,
\begin{align*}
\sum_{0<\gamma\leq T}\prod^{\mathcal{J}}_{j=1} \big ( |{\mathcal N}_j(\rho, k)|^2+ |{\mathcal Q}_j(\rho,k)|^{2r_k} \big ) \ll & T ( \log T  )^{k^2+1}.
\end{align*}
\end{proposition}

\begin{proposition}
\label{Prop6}
  With notations as above, we have for $0<k \leq 1/2$,
\begin{align*}
\sum_{0<\gamma\leq T}|\zeta'(\rho)|^{2} |\mathcal{N}(\rho, k-1)|^2   \ll T ( \log T  )^{k^2+3}.
\end{align*}
\end{proposition}

  We shall prove the above propositions in the rest of the paper.

\subsection{Proof of Proposition \ref{Prop4}}
\label{sec 4.8}

  The proof follows largely the arguments in Section 5 of \cite{MN1}.  We begin by recalling a few results on the Riemann zeta function $\zeta(s)$. Notice that $\zeta(s)$ satisfies the functional equation (see \cite[Corollary 10.4]{MVa}):
\begin{align}
\label{fcneqn}
   \zeta(s) = \chi(s) \zeta(1-s),
\end{align}
 where
\[
   \chi(s) =  2^s\pi^{s-1} \Gamma(1-s) \sin (\pi s/2) \ .
\]

  Logarithmically differentiating the functional equation above implies that
\begin{align}
\label{logzetader}
   \frac{\zeta'}{\zeta}(1-s) = \frac{\chi'}{\chi}(s) - \frac{\zeta'}{\zeta}(s).
\end{align}

 Moreover, we have (see \cite[(8)]{MN1}) uniformly for $-1 \leq \sigma \leq 2$ and $|t| \geq 1$,
\begin{align}
\label{chiderest}
  \frac{\chi'}{\chi}(\sigma+it)=\frac{\chi'}{\chi}(1-\sigma-it)=-\log \leg {|t|}{2\pi}+O(\frac 1{|t|}).
\end{align}

  Note that it follows from \cite[p. 108]{Da} that for every $t \geq 2$ and all nontrivial zeros $\rho =\half + i\gamma$ of $\zeta(s)$, there exists a number $T$ satisfying $t \leq T \leq t + 1$ such that
\begin{align}
\label{Tcond}
 \frac{\zeta'}{\zeta}(\sigma+it) \ll (\log T)^2, \quad \text{for} -1 \leq \sigma \leq 2 \ \ \text{and} \ \ |\gamma-T| \gg (\log T)^{-1}.
\end{align}

  Note also that by differentiating the functional equation \eqref{fcneqn} above, we have
\begin{align}
\label{zetader}
  \zeta'(s) = - \chi(s) \left( \zeta'(1-s) - \frac{\chi'}{\chi}(s) \zeta(1-s) \right).
\end{align}
  We deduce from this that if we denote the left side expression in \eqref{L1estmation} by $S$, then
\begin{align*}
   S
   & = -\sum_{0 < \gamma \leq T} \chi(\rho) \zeta'(1-\rho) \mathcal{N}(\rho, k-1)\mathcal{N}(1-\rho, k) = \frac{1}{2 \pi i} \int_{\mathcal{C}} \frac{\zeta'}{\zeta}(1-s) \chi(s) \zeta'(1-s)  \mathcal{N}(s, k-1)\mathcal{N}(1-s, k) \, ds,
\end{align*}
where $\mathcal{C}$ is the positively oriented rectangle with vertices at $1-\kappa+i,\kappa+i,\kappa+iT,$ and $1-\kappa+iT$, and
$\kappa=1+(\log T)^{-1}$.  Here we may choose $T$ to satisfy the conditions in \eqref{Tcond}.   We then apply \eqref{Nnormbound}, \eqref{Tcond} and the estimations
\begin{equation*}
\begin{split}
   \chi(s)  \ll T^{1/2-\sigma} , \quad \zeta'(1-s)  \ll  T^{\sigma/2+\epsilon} \
\end{split}
\end{equation*}
  to see that the integral is bounded by $O(T^{1-\varepsilon})$ on the horizontal edges of the contour.

Next, we deduce from \eqref{zetader} and the functional equation \eqref{fcneqn} that
\[
  \chi(s) \zeta'(1-s) = - \zeta'(s) + \frac{\chi'}{\chi}(s) \zeta(s).
\]
  Combining this with \eqref{logzetader}, we see that the integral on the right edge of the contour equals
\begin{equation*}
  S_{R} = \frac{1}{2 \pi i} \int_{\kappa+i}^{\kappa+iT} \left(
  \frac{\chi'}{\chi}(s)^{2} \zeta(s) - 2 \frac{\chi'}{\chi}(s) \zeta'(s)
  + \frac{\zeta'}{\zeta}(s) \zeta'(s)
  \right) \mathcal{N}(s, k-1)\mathcal{N}(1-s, k) \, ds.
\end{equation*}
  Also, the integral on the left edge of the contour equals
\[
  S_{L} = \frac{1}{2 \pi i} \int_{1-\kappa+iT}^{1-\kappa+i} \frac{\zeta'}{\zeta}(1-s) \chi(s) \zeta'(1-s)\mathcal{N}(s, k-1)\mathcal{N}(1-s, k) \, ds.
\]
 we make a change of variable $s \to 1-s$ to see that $S_L=-\overline{I}_{L}$, where
\[
   I_{L} = \frac{1}{2 \pi i} \int_{\kappa+i}^{\kappa+iT}  \chi(1-s) \frac{\zeta'}{\zeta}(s) \zeta'(s) \mathcal{N}(1-s, k-1)\mathcal{N}(s, k)
   \, ds.
\]

  We then conclude that
\[
   S = S_{R}- \overline{I}_{L} + O(T^{1-\varepsilon}).
\]

  We now apply  \eqref{Nnormbound}, \eqref{chiderest} and the bounds (see \cite[Corollary 1.17, Theorem 6.7]{MVa}) that when $s=\kappa+it$,
\begin{align*}
  \zeta(s) \ll & (1+(|t|+4)^{1-\kappa})\min (\frac 1{|\kappa-1|}, \log (|t|+4))+\frac 1{s-1}+O(1), \\
  \frac{\zeta'}{\zeta}(s)  \ll & \log (|t|+4)-\frac 1{s-1}+O(1),
\end{align*}
  to see that
\begin{align}
\label{SRint}
  S_{R} = \frac{1}{2 \pi i} \int_{\kappa+i}^{\kappa+iT} \left(
  \log^2 \left( \frac{t}{2\pi } \right ) \zeta(s) +2\log \left( \frac{t}{2\pi } \right ) \zeta'(s)
  + \frac{\zeta'}{\zeta}(s) \zeta'(s)
  \right) \mathcal{N}(s, k-1)\mathcal{N}(1-s, k) \, ds+O(T^{1-\varepsilon}).
\end{align}

  To evaluate $S_{R}$, we define the Dirichlet convolution $f*g$ for two arithmetic functions $f(k), g(k)$ by
\begin{align*}
  f*g(k)=\sum_{mn=k}f(m)g(n).
\end{align*}

   We then denote the integral given in \eqref{SRint} as a sum of three terms: $S_{R,1}, S_{R,2}$ and $S_{R,3}$, where
\begin{align*}
  S_{R,1} =& \frac{1}{2 \pi i} \int_{\kappa+i}^{\kappa+iT}
  \log^2 \left( \frac{t}{2\pi } \right ) \zeta(s) \mathcal{N}(s, k-1)\mathcal{N}(1-s, k)\, ds,  \\
  S_{R,2} =& \frac{2}{2 \pi i} \int_{\kappa+i}^{\kappa+iT}\log \left( \frac{t}{2\pi } \right ) \zeta'(s)\mathcal{N}(s, k-1)\mathcal{N}(1-s, k) \, ds, \\
  S_{R,3} =& \frac{1}{2 \pi i} \int_{\kappa+i}^{\kappa+iT} \frac{\zeta'}{\zeta}(s) \zeta'(s)\mathcal{N}(s, k-1)\mathcal{N}(1-s, k) \, ds.
\end{align*}

  We use the notation given in \eqref{Nskexpression} and apply Lemma \ref{Lem-MVDP} to evaluate $S_{R,1}$ to obtain that
\begin{align*}
  S_{R,1} =&  \left ( \frac{1}{2 \pi} \int_{1}^{T} \left(
  \log^2  \frac{t}{2\pi } \right )dt \right )\sum_{n}\frac {1 * a_{k-1}(n) \cdot a_k(n)}{n}\\
  &+ \left ( \left (\log^2 T +\int^T_1| (\log^2  \frac{t}{2\pi })'|dt \right ) \left (\sum^{\infty}_{n=1} \frac {(1 * a_{k-1})(n)^2}{n^{2\kappa-1}}\right )^{\half}\left (\sum^{\infty}_{n=1} \frac {a_{k}(n)^2}{n^{1-2\kappa}}\right )^{\half}\right ).
\end{align*}

   Similarly, we have
\begin{align*}
  S_{R,2} =&  -2\left ( \frac{1}{2 \pi} \int_{1}^{T} \left(
  \log \frac{t}{2\pi } \right )dt \right )\sum_{n}\frac {\log * a_{k-1}(n) \cdot a_k(n)}{n}\\
&+ \left ( \left (\log T +\int^T_1| (\log  \frac{t}{2\pi })'|dt \right ) \left (\sum^{\infty}_{n=1} \frac {(\log * a_{k-1})(n)^2}{n^{2\kappa-1}}\right )^{\half}\left (\sum^{\infty}_{n=1} \frac {a_{k}(n)^2}{n^{1-2\kappa}}\right )^{\half}\right ).
\end{align*}

   Also,
\begin{align*}
  S_{R,3} =&  \left ( \frac{1}{2 \pi} \int_{1}^{T} 1 dt \right )\sum_{n}\frac {(\Lambda*\log)* a_{k-1}(n) \cdot a_k(n)}{n}+ \left ( \left (\sum^{\infty}_{n=1} \frac {(\Lambda*\log)* a_{k-1}(n)^2}{n^{2\kappa-1}}\right )^{\half}\left (\sum^{\infty}_{n=1} \frac {a_{k}(n)^2}{n^{1-2\kappa}}\right )^{\half}\right ).
\end{align*}

  We now apply the estimations given in \eqref{anbound} to see that for $T$ large enough,
\begin{align*}
 \sum^{\infty}_{n=1} \frac {a_{k}(n)^2}{n^{1-2\kappa}} \ll e^{4k\log T/\log \log T}\sum_{n \leq T^{40 e^2k10^{-M/4}}} \frac {1}{n^{1-2\kappa}}\ll T^{1-\varepsilon}.
\end{align*}

  Moreover, using the estimation $(\Lambda*\log)(n) \leq \log n \sum_{d|n}\Lambda(d) =\log^2n$, we see that
\begin{align}
\label{convbounds}
\begin{split}
 & 1* a_{k-1}(n) \leq \sum_{n \leq T^{40 e^2k10^{-M/4}}}a_{k-1}(n) \leq T^{1/2-\varepsilon}, \\
 & (\log)* a_{k-1}(n) \leq \log n \sum_{n \leq T^{40 e^2k10^{-M/4}}}a_{k-1}(n) \leq T^{1/2-\varepsilon}\log n, \\
 & (\Lambda*\log)* a_{k-1}(n) \leq  \log^2 n \sum_{n \leq T^{40 e^2k10^{-M/4}}}a_{k-1}(n) \leq T^{1/2-\varepsilon}\log^2 n.
\end{split}
\end{align}
  It follows that
\begin{align*}
  \sum^{\infty}_{n=1} \frac {1* a_{k-1}(n)^2+(\log)* a_{k-1}(n)^2+(\Lambda*\log)* a_{k-1}(n)^2}{n^{2\kappa-1}} \ll T^{1-2\varepsilon}\sum^{\infty}_{n=1} \frac {\log^4 n}{n^{2\kappa-1}} \ll T^{1-\varepsilon},
\end{align*}
  where the last estimation above follows from the bound that (see \cite[(16)]{MN1}) uniformly for $\sigma>1$ and any integer $i \geq 0$,
\begin{align}
\label{lognbound}
  \sum^{\infty}_{n=1} \frac {\log^i n}{n^{\sigma}} \ll \frac 1{(\sigma-1)^{i+1}}.
\end{align}

   We apply the above estimations in the evaluations of $S_{R,1}, S_{R,2}$ and $S_{R,3}$ to see that the contributions from the error terms can be ignored. Furthermore, in the evaluation of $S_{R,1}$, we see that, for a monic polynomial $\mathcal{Q}_2$ of degree $2$,
\begin{align*}
   \frac{1}{2 \pi} \int_{1}^{T} \left(
  \log^2  \frac{t}{2\pi } \right )dt = \frac{T}{2\pi }\mathcal{Q}_2(\mathcal{L})+O(1),
\end{align*}
    where we denote $\cL=\log (T/(2\pi))$.

   Now, error term above contributes to an negligible error term since by \eqref{anbound} and \eqref{convbounds}, we have
\begin{align*}
  \sum_{n}\frac {1 * a_{k-1}(n) \cdot a_k(n)}{n} \ll T^{1/2-\varepsilon}\sum_{n \leq T^{40 e^2k10^{-M/4}}}\frac {1}{n} \ll T^{1-\varepsilon}.
\end{align*}

  We treat the integrals in the expressions for $S_{R,2}$ and $S_{R,3}$ similarly to arrive that, for monic polynomials $\mathcal{Q}_i, i=1,2$ of degree $i$,
\begin{align}
\label{SRexp}
  S_{R} =&  \frac{T}{2\pi }\sum_{n, m}\frac {a_{k-1}(m) a_k(mn)}{mn}\left ( \mathcal{Q}_2(\mathcal{L})-2\mathcal{Q}_1(\mathcal{L})(\log n)+((\Lambda*\log)(n)\right )+O(T^{1-\varepsilon}).
\end{align}

  Next, we evaluate $I_L$ using the notation given in \eqref{Nskexpression} to see that
\begin{align}
\label{IL}
   I_L=\sum_{m}\sum_{n}\frac {a_{k-1}(m) a_k(n)}{m} \sum_{k}(\Lambda * \log )(k) \left ( \frac{1}{2 \pi i} \int_{\kappa+i}^{\kappa+iT}  \chi(1-s) \leg{n k}{m}^{-s} \right )
   \, ds  \ .
\end{align}

   To evaluate the integral above, we need the following result from \cite[Lemma 5.2]{MN1}.
\begin{lemma} \label{stph} Let $r, \kappa_0 >0$. We have uniformly for $\kappa_0 \le \kappa \le 2$ that
\[
    \frac{1}{2 \pi i} \int_{\kappa+i}^{\kappa+iT} \chi(1-s) r^{-s} \, ds
    = I_{[0,T/ 2 \pi]} (r) e(r) + O\left (r^{-\kappa} \left(T^{\kappa-1/2}+ \frac{T^{\kappa+1/2}}{|T-2 \pi r|+T^{1/2}} \right )\right ).
\]
 where we write $e(x)=e^{2 \pi i x}$ and we denote $I_{[0,T/ 2 \pi]} (r)$ for  the indicator function of the
interval $[0, \frac T{2\pi}]$, namely, $I_{[0,T/ 2 \pi]} (r)=1$ if $r \in [0, \frac T{2\pi}]$ and $I_{[0,T/ 2 \pi]} (r)=0$
otherwise.
\end{lemma}

  We apply Lemma \ref{stph} to see that the contribution from the error term in \eqref{IL} is
\begin{align*}
   \ll \sum_{m}\sum_{n}\frac {a_{k-1}(m) a_k(n)}{m} \sum_{k}(\Lambda * \log )(k) \leg{n k}{m}^{-\kappa} \left(T^{\kappa-1/2}+ \frac{T^{\kappa+1/2}}{|T-2 \pi \leg{n k}{m}|+T^{1/2}} \right ).
\end{align*}

   Using \eqref{lognbound}, we see that
\begin{align*}
  \sum^{\infty}_{k=1} \frac {(\Lambda * \log )(k) }{k^{\kappa}} \leq  \sum^{\infty}_{k=1} \frac {\log^2(k) }{k^{\kappa}} \ll \frac 1{(\kappa-1)^3}.
\end{align*}

   We deduce from this and \eqref{anbound} that
\begin{align*}
   \ll \sum_{m}\sum_{n}\frac {a_{k-1}(m) a_k(n)}{m} \sum_{k}(\Lambda * \log )(k) \leg{n k}{m}^{-\kappa} T^{\kappa-1/2} =O(T^{1-\varepsilon}).
\end{align*}

   We now use the ideas in the proof of \cite[Lemma 2]{CGG} to estimate
\begin{align*}
   \sum_{m}\sum_{n}\frac {a_{k-1}(m) a_k(n)}{m} \sum_{k}(\Lambda * \log )(k) \leg{n k}{m}^{-\kappa} \frac{T^{\kappa+1/2}}{|T-2 \pi \leg{n k}{m}|+T^{1/2}}.
\end{align*}
  We break up the sum into three parts. The terms with $|T-2 \pi \leg{n k}{m}| >\half T$ contribute $O(T^{1-\varepsilon})$ as our discussions above. The terms $T^{\half} \leq |T-2 \pi \leg{n k}{m}| \leq \half T$ are further divided into cases that $T^{\half} \leq T-2 \pi \leg{n k}{m} \leq \half T$ and  $T^{\half} \leq 2 \pi \leg{n k}{m}-T \leq \half T$. Without loss of generality, we consider the case that $T^{\half} \leq 2 \pi \leg{n k}{m}-T \leq \half T$. This implies that $T+T^{\half}  \leq 2 \pi \leg{n k}{m}\leq T+\half T $. We then split the sum into $\ll \log T$ sums of the shape
\begin{align*}
   T+P  \leq 2 \pi \leg{n k}{m}\leq T+2P
\end{align*}
  where $T^{\half} \ll P \ll T$. The above implies that
\begin{align}
\label{nkbound}
  (nk)^{-\kappa} \ll (mT)^{-\kappa}.
\end{align}
   Moreover, the sum over $n, k$ ranges over an interval of length $\ll mP$, Thus, the contribution form the corresponding terms (using $a_k(n) \ll T^{\varepsilon}$, $(\Lambda * \log )(k) \ll \log^2 k \ll \log^2 (m(T+P)) \ll T^{\varepsilon}$, $1/(|T-2 \pi \leg{n k}{m}|+T^{1/2}) \ll 1/P$) is
\begin{align*}
   T^{\varepsilon}\sum_{m}\sum_{n}\frac {a_{k-1}(m)}{m} mP (mT)^{-\kappa} \frac{T^{\kappa+1/2}}{P} \ll T^{1-\varepsilon}.
\end{align*}

   Lastly, we consider the contributions from the terms $|T-2 \pi \leg{n k}{m}| \leq T^{\half}$ by noticing that in this case the estimation \eqref{nkbound} is still valid. Moreover, the sum over $n, k$ ranges over an interval of length $\ll mT^{\half}$, Thus, the contribution form the corresponding terms is
\begin{align*}
   \ll T^{\varepsilon}\sum_{m}\sum_{n}\frac {a_{k-1}(m)}{m} mT^{\half} (mT)^{-\kappa} \frac{T^{\kappa+1/2}}{T^{\half}} \ll T^{1-\varepsilon}.
\end{align*}

   We then conclude from the above discussions that
\begin{align*}
   I_L=\sum_{m}\sum_{n}\frac {a_{k-1}(m) a_k(n)}{m} \sum_{1 \leq k \leq \frac {mT}{2\pi n}}e(-\frac {kn}{m})+O(T^{1-\varepsilon}).
\end{align*}

   We now proceed as in \cite[p. 3212-3213]{MN1} to see that, for a monic polynomial $\mathcal{P}_2$ of degree $2$,
\begin{align}
\label{ILexp}
\begin{split}
   I_{L} =&  \frac{T}{4\pi }\sum_{n,m}\frac {a_{k-1}(m) a_k(mn)}{mn}\mathcal{P}_2(\log \frac {T}{2\pi n})+O(T^{1-\varepsilon})+O(\sum_{n,m}\frac {a_{k-1}(m) a_k(n)}{mn})(m,n)\Big ( \Lambda_2(\frac {m}{(m,n)})+\Lambda(\frac {m}{(m,n)})\log T \Big ).
\end{split}
\end{align}

   To estimate the last error term in \eqref{ILexp}, we write $d=(m,n), m=dL, n=dN$ and use the easily checked property that $a_{k-1}(mn) \leq a_{k-1}(m)a_{k-1}(n)$ to see that it is
\begin{align*}
   \ll \sum_{d,N}\frac {a_{k-1}(d) a_k(dN)}{dN}\sum_{L}\frac {a_{k-1}(L)}{L}\Big ( \Lambda_2(L)+\Lambda(L)\log T \Big ).
\end{align*}

   Using \eqref{Lambda2} and the observation from \eqref{anbound} that $a_{k-1}(L)$ is bounded when $L$ is supported on integers $L$ with $\omega(L) \leq 2$ and $a_{k-1}(L) \neq 0$ only for $L \leq T^{40 e^2k10^{-M/4}}$, we proceed as in \cite[p. 3213]{MN1} to see  that
\begin{align*}
   \sum_{L}\frac {a_{k-1}(L)}{L}\Big ( \Lambda_2(L)+\Lambda(L)\log T \Big ) \ll 10^{-M/4}\log^2 T.
\end{align*}

   We conclude from \eqref{SRexp}, \eqref{ILexp} and the above estimation that
\begin{align*}
  S \geq &  \frac{T}{2\pi }\sum_{n,m}\frac {a_{k-1}(m) a_k(mn)}{mn}\left ( \mathcal{Q}_2(\mathcal{L})-2\mathcal{Q}_1(\mathcal{L})(\log n)-\half \mathcal{P}_2(\mathcal{L}-\log n)+((\Lambda*\log)(n)\right ) \\
  &+O(10^{-M/4}T\log^2 T\sum_{n,m}\frac {a_{k-1}(m) a_k(mn)}{mn})+O(T^{1-\varepsilon}).
\end{align*}

   We now take $M$ large enough and argue as in \cite[p. 3214]{MN1} to see that
\begin{align*}
  S \gg & T\log^2 T \sum_{n,m}\frac {a_{k-1}(m) a_k(mn)}{mn}.
\end{align*}

   Lastly, we apply arguments used in \cite[Section 4]{Gao2021-4} to estimate the sums above to arrive that
\begin{align*}
  S \gg & T(\log T)^{k^2+2} .
\end{align*}
   This completes the proof of the proposition.

\subsection{Proof of Proposition \ref{Prop5}}

   Upon dividing the range of $\gamma$ into dyadic blocks and replacing $T$ by $2T$, we see that it suffices to show for large $T$,
\begin{align}
\label{sumprodNQ}
\sum_{T<\gamma\leq 2T}\prod^{\mathcal{J}}_{j=1} \big ( |{\mathcal N}_j(\rho, k)|^2+ |{\mathcal Q}_j(\rho,k)|^{2r_k} \big ) \ll & T ( \log T  )^{k^2+1}.
\end{align}
    We apply Lemma \ref{Lem-Landau} to evaluate the left side expression above. In this process, we may ignore the contributions from the error terms in \eqref{sumgamma}, using arguments similar to our treatments on various error terms in the proof of Proposition \ref{Prop4}. Moreover, we may also ignore the contributions from the main terms in \eqref{sumgamma} from the cases $a \neq b$, since these terms are negative and we are seeking for an upper bound here. Thus, only the diagonal terms in the left side expression of \eqref{sumprodNQ} survive. Now applying \eqref{N2Tbound}, we conclude that
\begin{align*}
\begin{split}
 &\sum_{T<\gamma\leq 2T}\prod^{\mathcal{J}}_{j=1} \big ( |{\mathcal N}_j(\rho, k)|^2+ |{\mathcal Q}_j(\rho,k)|^{2r_k} \big ) \\
\ll & T ( \log T  )\prod^{\mathcal{J}}_{j=1} \Big (  \sum_{n_j} \frac{k^{2\Omega(n_j)}}{n_j g^2(n_j)}  b_j(n_j)  + \Big( \frac{64 \max (2, k+3/2 ) }{\lceil e^2k\alpha^{-3/4}_j \rceil} \Big)^{2r_k \lceil e^2k\alpha^{-3/4}_j \rceil}((r_k\lceil e^2k\alpha^{-3/4}_j \rceil)!)^2 \sum_{ \substack{ \Omega(n_j) = r_k\lceil e^2k\alpha^{-3/4}_j \rceil \\ p|n_j \implies  p\in I_j}} \frac{1 }{n_j g^2(n_j)} \Big ).
\end{split}
\end{align*}

  The product above is analogues to the product in \cite[(6.20)]{Gao2021-4}. Using similar estimations, we see that it is $\ll (\log T)^{k^2}$. This implies \eqref{sumprodNQ} and completes the proof of the proposition.

\subsection{Proof of Proposition \ref{Prop6}}
\label{sec: proof of Prop 6}

   Again by dividing the range of $\gamma$ into dyadic blocks and replacing $T$ by $2T$, we see that it suffices to show for large $T$,
\begin{align}
\label{zetaNestmation}
\sum_{T<\gamma\leq 2T}|\zeta'(\rho)|^{2} |\mathcal{N}(\rho, k-1)|^2   \ll T ( \log T  )^{k^2+3}.
\end{align}

    We split the interval $(0,T^{\alpha_{\mathcal{J}}}]$ into disjoint subintervals $I_j=(T^{\alpha_{j-1}},T^{\alpha_j}]$ for $1\le j \le \mathcal{J}$ and define
$$w_j(n)=\frac{\Lambda_{\cL}(n)}{n^{1/(\alpha_j\log T)}\log n}\frac{\log (T^{\alpha_j}/n)}{\log T^{\alpha_j}},$$
 where
\begin{align*}
  \Lambda_{\cL}(n) =
\begin{cases}
  \Lambda(n),  \quad \text{ if } n=p \text{ or if } n=p^2 \text{ and } n \leq \cL, \\
   0,  \quad \text{otherwise}.
\end{cases}
\end{align*}

 For $1 \leq l \leq j \leq \mathcal{J}$, we define
\[
G_{l,j}(t)=\Re\sum_{n\in I_l}\frac{w_j(n)}{\sqrt{n}}n^{-it}.
\]

  Note that our definition of $G_{l,j}$ is slightly different from that in \cite{Kirila}, due to our definition on $\cL$. However, we notice the bounds
\begin{align}
\label{wbounds}
\begin{split}
  w_j(p) \leq 1, \quad w_j(p^2) \leq \half, 
\end{split}
\end{align}
 so that it follows from the above and Lemma \ref{RS} that
$$\sum_{\cL \leq p \leq \log T}\frac{w_j(p^2)}{p} =O(1).$$
 We then deduce from \cite[(4.1)]{Kirila} the following upper bound for $\log |\zeta'(\rho)|$, which says that for any $1 \leq j \leq \mathcal{J}$,
\begin{align}
\label{basicest}
\begin{split}
 & \log |\zeta'(\rho)| \ll \sum^{j}_{l=1}G_{l,j}(\gamma)+\log \log T+ \alpha^{-1}_j+O(1).
\end{split}
 \end{align}

  We also define the following sets:
\begin{align*}
  \mathcal{S}(0) =& \{ T< \gamma \leq 2T : |G_{1,l}(\gamma)| > \alpha_{1}^{-3/4} \; \text{ for some } 1 \leq l \leq \mathcal{J} \} ,   \\
 \mathcal{S}(j) =& \{ T< \gamma \leq 2T : |  G_{m,l}(\gamma)| \leq
 \alpha_{m}^{-3/4} \; \forall 1 \leq m \leq j, \; \forall m \leq l \leq \mathcal{J}, \\
 & \;\;\;\;\; \text{but }  |G_{j+1,l}(\gamma)| > \alpha_{j+1}^{-3/4} \; \text{ for some } j+1 \leq l \leq \mathcal{J} \} ,  \quad  1\leq j \leq \mathcal{J}, \\
 \mathcal{S}(\mathcal{J}) =& \{T< \gamma \leq 2T : |G_{m,
\mathcal{J}}(\gamma)| \leq \alpha_{m}^{-3/4} \; \forall 1 \leq m \leq \mathcal{J}\},
\end{align*}
  so that
$$ \{ T< \gamma \leq 2T   \}=\bigcup_{j=0}^{ \mathcal{J}} \mathcal{S}(j). $$

  It follows that
\begin{align}
\label{sumovermj'}
  \sum_{T<\gamma\leq 2T}|\zeta'(\rho)|^{2} |\mathcal{N}(\rho, k-1)|^2 = \sum_{j=0}^{ \mathcal{J}}  \sum_{\gamma \in S(j)}|\zeta'(\rho)|^{2} |\mathcal{N}(\rho, k-1)|^2.
\end{align}

   We note from \cite[Lemma 5.5]{Kirila} and \eqref{N2Tbound} that we have
\begin{align}
\label{S0bound}
\text{meas}(\mathcal{S}(0)) \ll &
T(\log T)e^{-(\log\log T)^{2}/10}  .
\end{align}

   We then deduce via the Cauchy-Schwarz inequality that
\begin{align}
\label{LS0bound}
\begin{split}
& \sum_{\gamma \in S(0)}|\zeta'(\rho)|^{2} |\mathcal{N}(\rho, k-1)|^2
\leq  \Big ( \text{meas}(\mathcal{S}(0)) \Big )^{1/4} \Big (
\sum_{T<\gamma\leq 2T}|\zeta'(\rho)|^{8}  \Big )^{1/4} \Big ( \sum_{T<\gamma\leq 2T} |\mathcal{N}(\rho, k-1)|^4  \Big )^{1/2}.
\end{split}
\end{align}

  Similar to the proof of Proposition \ref{Prop5}, we have that
\begin{align}
\label{N2k2bound}
&  \sum_{T<\gamma\leq 2T} |\mathcal{N}(\rho, k-1)|^4 \ll T( \log T  )^{(2k-2))^2+1}.
\end{align}

  Also, applying \eqref{Jupperbound} with $k=4, \varepsilon=1$, we see that
\begin{align}
\label{L8bound}
\sum_{T<\gamma\leq 2T}|\zeta'(\rho)|^{8}  \leq T(\log T)^{25}.
\end{align}

  We use the bounds given in \eqref{S0bound}, \eqref{N2k2bound} and \eqref{L8bound} in \eqref{LS0bound} to conclude that
\begin{align}
\label{S0bound1}
\sum_{\gamma \in \mathcal{S}(0)}|\zeta'(\rho)|^{2} |\mathcal{N}(\rho, k-1)|^2   \ll T(\log T)^{k^2+3}.
\end{align}

  The above estimation implies that it remains to consider the cases $j \geq 1$ in \eqref{sumovermj'}. Without loss of generality, we may assume that $1 \leq j \leq \mathcal{J}-1$ here. When $\gamma \in \mathcal{S}(j)$, we deduce from \eqref{basicest} that
\begin{align*}
\begin{split}
 & \sum_{\gamma \in \mathcal{S}(j)}|\zeta'(\rho)|^{2}|\mathcal{N}(\rho, k-1)|^2 \ll (\log T)^{2} \exp \big(\frac {2}{\alpha_j} \big )\sum_{\gamma \in \mathcal{S}(j)} \exp \Big (
 2 \sum^{j}_{l=1} G_{l,j}(\gamma) \Big )|\mathcal{N}(\rho, k-1)|^2.
\end{split}
 \end{align*}

   As we have $G_{l, j} \leq  \alpha^{-3/4}_l$ when $\gamma \in \mathcal{S}(j)$, we apply \cite[Lemma 5.2]{Kirila} to see that
\begin{align*}
\begin{split}
\exp \Big (
 2 \sum^{j}_{l=1} G_{l,j}(\gamma) \Big ) \ll
\prod^j_{l=1}E^2_{e^2k\alpha^{-3/4}_l}(G_{l,j}(\gamma)).
\end{split}
 \end{align*}

   We then deduce from the description on $\mathcal{S}(j)$ that when $j \geq 1$,
\begin{align}
\label{upperboundprodE0}
\begin{split}
 & \sum_{\gamma \in \mathcal{S}(j)}|\zeta'(\rho)|^{2}|\mathcal{N}(\rho, k-1)|^2   \\
 \ll &  (\log T)^2 \exp \big(\frac {2}{\alpha_j} \big )
 \sum^{ \mathcal{I}}_{m=j+1} \sum_{\gamma \in \mathcal{S}(j)} \exp \Big (
 2 \sum^{j}_{l=1} G_{l,j}(\gamma) \Big )|\mathcal{N}(\rho, k-1)|^2 \Big ( \alpha^{3/4}_{j+1}G_{j+1,m}(\gamma)\Big)^{2\lceil 1/(10\alpha_{j+1})\rceil } \\
\ll & (\log T)^2 \exp \big(\frac {2}{\alpha_j} \big )
 \sum^{ \mathcal{I}}_{m=j+1} \sum_{\gamma \in \mathcal{S}(j)}
\prod^j_{l=1} E^2_{e^2k\alpha^{-3/4}_l}(G_{l,j}(\gamma))|E_{e^2k\alpha^{-3/4}_l}((k-1){\mathcal P}_{l}(\rho)|^2 \\
& \times |E_{e^2k\alpha^{-3/4}_{j+1}}((k-1){\mathcal P}_{j+1}(\rho)|^2  \Big ( \alpha^{3/4}_{j+1}G_{j+1,m}(\gamma)\Big)^{2\lceil 1/(10\alpha_{j+1})\rceil }  \prod^{\mathcal{J}}_{n=j+2} |E_{e^2k\alpha^{-3/4}_n}((2k-2){\mathcal P}_{n}(\rho))|^2.
\end{split}
\end{align}

  Note that
\begin{align}
\label{upperboundprodE}
\begin{split}
& \sum_{\gamma \in \mathcal{S}(j)}
\prod^j_{l=1} E^2_{e^2k\alpha^{-3/4}_l}(G_{l,j}(\gamma))|E_{e^2k\alpha^{-3/4}_l}((k-1){\mathcal P}_{l}(\rho)|^2 \\
& \times |E_{e^2k\alpha^{-3/4}_{j+1}}((k-1){\mathcal P}_{j+1}(\rho)|^2  \Big ( \alpha^{3/4}_{j+1}G_{j+1,m}(\gamma)\Big)^{2\lceil 1/(10\alpha_{j+1})\rceil }  \prod^{\mathcal{J}}_{n=j+2} |E_{e^2k\alpha^{-3/4}_n}((2k-2){\mathcal P}_{n}(\rho))|^2 \\
\leq & \sum_{T< \gamma \leq 2T}
\prod^j_{l=1} E^2_{e^2k\alpha^{-3/4}_l}(G_{l,j}(\gamma))|E_{e^2k\alpha^{-3/4}_l}((k-1){\mathcal P}_{l}(\rho)|^2 \\
& \times |E_{e^2k\alpha^{-3/4}_{j+1}}((k-1){\mathcal P}_{j+1}(\rho)|^2  \Big ( \alpha^{3/4}_{j+1}G_{j+1,m}(\gamma)\Big)^{2\lceil 1/(10\alpha_{j+1})\rceil }  \prod^{\mathcal{J}}_{n=j+2} |E_{e^2k\alpha^{-3/4}_n}((2k-2){\mathcal P}_{n}(\rho))|^2.
\end{split}
\end{align}

  We shall apply Lemma \ref{Lem-Landau} to evaluate the last sum above. As in the case for the proof of Proposition \ref{Prop4}, we may only focus on the main term in the process.  To facilitate our evaluation of the last sum above, we follow the treatments in \cite{Kirila} by introducing a sequence of independent random variables  $\{X_p\}$ such that each $X_p$ is uniformly distributed on the unit circle in the complex plane. We also define
\[
X_n=X_{p_1}^{h_1}\cdots X_{p_r}^{h_r}
\]
for $n=p_1^{h_1}\cdots p_r^{h_r}$ so that $X_n$ is a random completely multiplicative function. We then define random models $G_{l,j}(X)$ for $1 \leq l \leq j \leq {\mathcal J}$ by
\[
G_{l,j}(X)=\Re \sum_{n\in I_l}\frac{w_j(n)}{\sqrt{n}}X_n, \quad {\mathcal P}_{j}(X)=\sum_{p \in I_j}\frac{1}{\sqrt{p}}X_p.
\]

 Similar to \cite[Lemma 5.3]{Kirila}, we have under RH and other than a negligible error term, for $1 \leq l \leq j \leq {\mathcal J}$ and non-negative integers $n_{l,1}, n_{l,2}, n_{l,3}$ satisfying $\max ( n_{l,1}/2, n_{l,2}, n_{l,3}) \leq  \max (\lceil e^2k\alpha^{-3/4}_l \rceil, 2\lceil 1/(10\alpha_{j+1})\rceil) $,
\begin{align*}
\begin{split}
 & \sum_{T< \gamma \leq 2T}\prod^{{\mathcal J}}_{l=1}G^{n_{l,1}}_{l,j}(\gamma){\mathcal P}^{n_{l,2}}_{l}(\rho){\mathcal P}^{n_{l,3}}_{l}(\overline{\rho})
 \leq N(T, 2T) \E \Big ( \prod^{{\mathcal J}}_{l=1}G^{n_{l,1}}_{l,j}(X){\mathcal P}^{n_{l,2}}_{l}(X){\mathcal P}^{n_{l,3}}_{l}(\overline{X})  \Big ).
\end{split}
\end{align*}

  We now proceed as in the proof of \cite[Lemma 5.5]{Kirila} to see that, under RH and other than a negligible error term, the last sum in \eqref{upperboundprodE} is
\begin{align*}
\begin{split}
  \ll & \Big ( \alpha^{3/4}_{j+1}\Big)^{2\lceil 1/(10\alpha_{j+1})\rceil } N(T, 2T)  \E \Big ( G_{j+1,m}(X)^{2\lceil 1/(10\alpha_{j+1})\rceil } \exp \Big ( 2\sum^{j}_{l=1}G_{l,j}(X)+2(k-1)\sum^{{\mathcal J}}_{n=1}\Re {\mathcal P}_{n}(X) \Big ) \Big ) \\
=&  \Big ( \alpha^{3/4}_{j+1}\Big)^{2\lceil 1/(10\alpha_{j+1})\rceil } N(T, 2T)  \E \Big ( \exp \Big ( 2G_{1,j}(X)+2(k-1)\Re {\mathcal P}_{1}(X) \Big ) \Big ) \prod^{j}_{l=2}\E \Big ( \exp \Big ( 2 G_{l,j}(X)+2(k-1)\Re {\mathcal P}_{l}(X) \Big ) \Big ) \\
& \times \E \Big ( G_{j+1,m}(X)^{2\lceil 1/(10\alpha_{j+1})\rceil } \exp \Big ( 2(k-1)\Re {\mathcal P}_{j+1}(X)  \Big )\Big ) \times \prod^{{\mathcal J}}_{l=j+2}\E \Big ( \exp \Big ( 2(k-1)\Re {\mathcal P}_{l}(X)  \Big ).
\end{split}
\end{align*}

  Similar to the evaluation done on \cite[p. 492]{Kirila}, we see that
\begin{align}
\label{upperboundprodE2}
\begin{split}
  \prod^{j}_{l=2}\E \Big ( \exp \Big ( 2 G_{l,j}(X)+2(k-1)\Re {\mathcal P}_{l}(X) \Big ) \Big ) =& \prod^{j}_{l=2}
 \prod_{p \in I_l}I_0\big (\frac {2w_j(p)}{\sqrt{p}}+\frac {2(k-1)}{\sqrt{p}} \big ), \\
\prod^{{\mathcal J}}_{l=j+2}\E \Big ( \exp \Big ( 2(k-1)\Re {\mathcal P}_{l}(X) \Big ) \Big ) =& \prod^{{\mathcal J}}_{l=j+2}
 \prod_{p \in I_l}I_0\big (2(k-1)\frac 1{\sqrt{p}} \big ),
\end{split}
\end{align}
  where $I_0(z)=\sum_{n=0}^{\infty}\frac{(z/2)^{2n}}{(n!)^2}$ is the modified Bessel function of the first kind.

  Next, using the arguments similar to those in \cite[(6.3)]{Kirila}, we have that
\begin{align*}
\begin{split}
 & \E \Big ( \exp \Big ( 2 G_{1,j}(X)+2(k-1)\Re {\mathcal P}_{1}(X) \Big ) \Big )\\
 =&
 \prod_{\substack{p \in I_1 \\ p>\log T}}I_0\big (\frac {2w_j(p)}{\sqrt{p}}+\frac {2(k-1)}{\sqrt{p}} \big ) \\
& \times \E \Big ( \exp \Big ( \sum_{p \leq \log T}\big(\frac {2w_j(p)}{\sqrt{p}}+  \frac {2(k-1)}{\sqrt{p}}\big )\Re X_p+\sum_{p \leq \log T} \frac {4 w_j(p^2)}{\sqrt{p}}(\Re X_p^2)- \sum_{p \leq \log T}\frac {4 w_j(p^2)}{\sqrt{p}} \Big )\Big ).
\end{split}
\end{align*}

  We apply the bounds given in \eqref{wbounds} and proceed as in \cite[p. 492-493]{Kirila} to deduce that the last expression above is
\begin{align*}
\begin{split}
  \ll \prod_{\substack{p \in I_1}}I_0\big (\frac {2w_j(p)}{\sqrt{p}}+\frac {2(k-1)}{\sqrt{p}} \big ).
\end{split}
\end{align*}

 Furthermore, we notice that
\begin{align}
\label{upperboundprodE5}
\begin{split}
 & \E \Big ( G_{j+1,m}(X)^{2\lceil 1/(10\alpha_{j+1})\rceil } \exp \Big ( 2(k-1)\Re {\mathcal P}_{j+1}(X)  \Big )\Big ) \\
=& \sum^{\infty}_{n=0}\E \Big ( G_{j+1,m}(X)^{2\lceil 1/(10\alpha_{j+1})\rceil } \frac { (2(k-1)\Re {\mathcal P}_{j+1}(X))^n}{n!} \Big ) \\
=& \sum^{\infty}_{n=0}\E \Big ( G_{j+1,m}(X)^{2\lceil 1/(10\alpha_{j+1})\rceil } \frac { (2(k-1)\Re {\mathcal P}_{j+1}(X))^{2n}}{(2n)!} \Big ) \\
\leq & \sum^{\infty}_{n=0} \E \Big (  \frac { (2\Re {\mathcal P}_{j+1}(X))^{2\lceil 1/(10\alpha_{j+1})\rceil+ 2n}}{(2n)!} \Big ),
\end{split}
\end{align}
 where the last estimation above follows from the observation that, upon using \eqref{wbounds} and noting that $0<k \leq 1/2$,  we have for any integer $n \geq 0$,
\begin{align}
\label{upperboundprodE6}
\begin{split}
& \E \Big ( G_{j+1,m}(X)^{2\lceil 1/(10\alpha_{j+1})\rceil } \frac { (2(k-1)\Re {\mathcal P}_{j+1}(X))^{2n}}{(2n)!} \Big )
\leq  \ \E \Big (  \frac { (2\Re {\mathcal P}_{j+1}(X))^{2\lceil 1/(10\alpha_{j+1})\rceil+ 2n}}{(2n)!} \Big ).
\end{split}
\end{align}

  Observe that for any positive integer $m$,
\[
\E[(\Re X_p)^{m}]=
\begin{cases}
\displaystyle \binom{2h}{h}2^{-m}&\hbox{if } m=2h,\\
\displaystyle 0 &\hbox{if } m \text{ is odd}.
\end{cases}
\]

 It follows from this that if we denote $\{ p | p \in I_{j+1}\}=\{ p_1, \cdots, p_s \}$, then we have for any positive integer $m$,
\begin{align*}
\begin{split}
 \E \Big (  (2\Re {\mathcal P}_{j+1}(X))^{2m} \Big ) =&  \sum_{\substack{m_1, m_2 \cdots, m_s \geq 0 \\ m_1+m_2+\cdots+m_s =m}}\binom {2m}{2m_1, \cdots, 2m_s}\prod^s_{i=1}(\frac{2}{\sqrt{p_i}})^{2m_i} \E(\Re X_p)^{2m_i} \\
  =&  2^{-2m}\sum_{\substack{m_1, m_2 \cdots, m_s \geq 0 \\ m_1+m_2+\cdots+m_s =m}}\binom {2m}{2m_1, \cdots, 2m_s}\prod^s_{i=1}\binom {2m_i}{m_i}(\frac{4}{p_i})^{m_i} \\
\leq & \frac {(2m)!}{2^{2m}m! }(\sum_{p \in I_{j+1}}\frac{4}{p})^{m}.
\end{split}
\end{align*}

   We apply the above estimation to the last expression in \eqref{upperboundprodE6} to deduce that
\begin{align}
\label{upperboundprodE8}
\begin{split}
& \E \Big ( G_{j+1,m}(X)^{2\lceil 1/(10\alpha_{j+1})\rceil } \frac { (2(k-1)\Re {\mathcal P}_{j+1}(X))^{2n}}{(2n)!} \Big ) \\
 \leq & \  \frac{(2\lceil 1/(10\alpha_{j+1})\rceil+ 2n)!}{2^{2\lceil 1/(10\alpha_{j+1})\rceil+ 2n}(\lceil 1/(10\alpha_{j+1})\rceil+ n)!(2n)!}\bigg(\sum_{p\in I_{j+1}}\frac{4}{p}\bigg)^{\lceil1/10\alpha_{j+1}\rceil +n}.
\end{split}
\end{align}

 Using the bounds
\begin{align}
\label{Stirling}
  (\frac ne)^n \leq n! \leq n(\frac ne)^n,
\end{align}
  we see that the last expression in \eqref{upperboundprodE8} is
\begin{align}
\label{upperboundprodE9}
\begin{split}
\ll \frac {(2\lceil 1/(10\alpha_{j+1})\rceil+ 2n)}{(2n)!} \bigg(\frac {\lceil 1/(10\alpha_{j+1})\rceil+ n}{e}\sum_{T^{\alpha_j}<p\le T^{\alpha_{j+1}}}\frac{4}{p}\bigg)^{\lceil 1/(10\alpha_{j+1})\rceil+ n}.
\end{split}
\end{align}

   Notice that when $n \geq 2\lceil 1/(10\alpha_{j+1})\rceil$, the expression above is
\begin{align*}
\begin{split}
\leq \frac {3n}{(2n)!} \bigg(\frac {3n}{2e}\sum_{T^{\alpha_j}<p\le T^{\alpha_{j+1}}}\frac{4}{p}\bigg)^{3n/2} \leq 3n (\frac {e}{2n} )^{2n} \bigg(\frac {60n}{e}\bigg)^{3n/2},
\end{split}
\end{align*}
  where the last estimation above follows from \eqref{sumpj} and \eqref{Stirling}. As the sum of the last term above over $n$ is convergent, we see that the contribution of these terms to the last sum of \eqref{upperboundprodE5} is $O(1)$ and may be ignored.

   Now for $\frac 14 \lceil 1/(10\alpha_{j+1})\rceil \leq n < 2\lceil 1/(10\alpha_{j+1})\rceil$, we apply \eqref{sumpj} to see that the expression in \eqref{upperboundprodE9} is
\begin{align*}
\begin{split}
\leq & \frac {6\lceil 1/(10\alpha_{j+1})\rceil}{(2n)!} \bigg(\frac {3\lceil 1/(10\alpha_{j+1})\rceil}{e}\sum_{T^{\alpha_j}<p\le T^{\alpha_{j+1}}}\frac{4}{p}\bigg)^{\lceil 1/(10\alpha_{j+1})\rceil} \cdot \bigg(\frac {5n}{e}\sum_{T^{\alpha_j}<p\le T^{\alpha_{j+1}}}\frac{4}{p}\bigg)^{n} \\
\leq & \frac {6\lceil 1/(10\alpha_{j+1})\rceil}{(2n)!} \bigg(\frac {120\lceil 1/(10\alpha_{j+1})\rceil}{e}\bigg)^{\lceil 1/(10\alpha_{j+1})\rceil} \cdot \bigg(\frac {200n}{e}\bigg)^{n}.
\end{split}
\end{align*}
  Applying \eqref{Stirling} again, we see that the sum of the last term above over $n$ is convergent so that the contribution of these terms to the last sum of \eqref{upperboundprodE5} is
\begin{align}
\label{upperboundprodE12}
\begin{split}
\ll  \lceil 1/(10\alpha_{j+1})\rceil \bigg(\frac {120\lceil 1/(10\alpha_{j+1})\rceil}{e}\bigg)^{\lceil 1/(10\alpha_{j+1})\rceil}.
\end{split}
\end{align}

  Lastly, for $n< \frac 14 \lceil 1/(10\alpha_{j+1})\rceil$, we apply \eqref{sumpj} and \eqref{Stirling} one more time to see that the expression in \eqref{upperboundprodE9} is
\begin{align*}
\begin{split}
\leq & \frac {5\lceil 1/(10\alpha_{j+1})\rceil/2}{(2n)!} \bigg(\frac {5\lceil 1/(10\alpha_{j+1})\rceil}{4e}\sum_{T^{\alpha_j}<p\le T^{\alpha_{j+1}}}\frac{4}{p}\bigg)^{5\lceil 1/(10\alpha_{j+1})\rceil/4} \\
\leq &  \frac {5\lceil 1/(10\alpha_{j+1})\rceil/2}{(2n)!} \bigg(\frac {200\lceil 1/(10\alpha_{j+1})\rceil}{4e}\bigg)^{5\lceil 1/(10\alpha_{j+1})\rceil/4}.
\end{split}
\end{align*}
  Upon summing over $n$, we see that the contribution of these terms to the last sum of \eqref{upperboundprodE5} is
\begin{align}
\label{upperboundprodE14}
\begin{split}
\ll  \lceil 1/(10\alpha_{j+1})\rceil  \bigg(\frac {200\lceil 1/(10\alpha_{j+1})\rceil}{4e}\bigg)^{5\lceil 1/(10\alpha_{j+1})\rceil/4}.
\end{split}
\end{align}

   We apply the bounds \eqref{upperboundprodE12} and \eqref{upperboundprodE14} in \eqref{upperboundprodE5}  to conclude that
\begin{align*}
\begin{split}
& \E \Big ( G_{j+1,m}(X)^{2\lceil 1/(10\alpha_{j+1})\rceil } \exp \Big ( 2(k-1)\Re {\mathcal P}_{j+1}(X)  \Big )\Big ) \ll  \lceil 1/(10\alpha_{j+1})\rceil  \bigg(\frac {120\lceil 1/(10\alpha_{j+1})\rceil}{e}\bigg)^{5\lceil 1/(10\alpha_{j+1})\rceil/4}.
\end{split}
\end{align*}

   We combine \eqref{upperboundprodE2}, \eqref{upperboundprodE5} and the last estimation above to see that, via using \eqref{wbounds},  the last expression in \eqref{upperboundprodE} is
\begin{align*}
\begin{split}
  \ll & \Big ( \alpha^{3/4}_{j+1}\Big)^{2\lceil 1/(10\alpha_{j+1})\rceil }\lceil 1/(10\alpha_{j+1})\rceil  \bigg(\frac {120\lceil 1/(10\alpha_{j+1})\rceil}{e}\bigg)^{5\lceil 1/(10\alpha_{j+1})\rceil/4} N(T, 2T)   \prod^{{\mathcal J}}_{l=1}
 \prod_{p \in I_l}I_0\big (\frac {2 w_j(p)}{\sqrt{p}}+\frac {2(k-1)}{\sqrt{p}} \big ) \\
\ll & \Big ( \frac 1{\alpha_{j+1}}\Big)^{-\lceil 1/(10\alpha_{j+1})\rceil/8 } N(T, 2T)   \prod^{{\mathcal J}}_{l=1}
 \prod_{p \in I_l}I_0\big (\frac {2k}{\sqrt{p}} \big ).
\end{split}
\end{align*}

  Using further the estimation that $I_0(2x) \leq e^{x^2}$, we deduce from the above that
the last expression in \eqref{upperboundprodE} is
\begin{align}
\label{upperboundprodE17}
\begin{split}
  \ll \Big ( \frac 1{\alpha_{j+1}}\Big)^{-\lceil 1/(10\alpha_{j+1})\rceil/8 } N(T, 2T) \exp \Big (\sum_{p \leq T^{\alpha_{{\mathcal J}}}}\frac {k^2}{p} \Big ) \ll  \Big ( \frac 1{\alpha_{j+1}}\Big)^{-\lceil 1/(10\alpha_{j+1})\rceil/8 }T(\log T)^{k^2+1},
\end{split}
\end{align}
  where the last estimation above follows from \eqref{upperboundprodE17} and \eqref{N2Tbound}.

 As $20/\alpha_{j+1}=1/\alpha_j$, we conclude from \eqref{sumovermj'}, \eqref{S0bound1}, \eqref{upperboundprodE0}, \eqref{upperboundprodE} and \eqref{upperboundprodE17} and that
\begin{align*}
\begin{split}
  \sum_{T<\gamma\leq 2T}|\zeta'(\rho)|^{2} |\mathcal{N}(\rho, k-1)|^2 \ll & \sum_{j=1}^{ \mathcal{J}}  \sum_{\gamma \in S(j)}|\zeta'(\rho)|^{2} |\mathcal{N}(\rho, k-1)|^2+T(\log T)^{k^2+3} \\
\ll & T(\log T)^{k^2+3}\sum_{j=1}^{ \mathcal{J}}\Big ( \frac 1{\alpha_{j+1}}\Big)^{-\lceil 1/(10\alpha_{j+1})\rceil/(10) }.
\end{split}
\end{align*}
   As the sum of the right side expression over $j$ converges, we see that the above estimation implies \eqref{zetaNestmation}
and this completes the proof of Proposition \ref{Prop6}.

\vspace*{.5cm}

\noindent{\bf Acknowledgments.} P. G. is supported in part by NSFC grant 11871082.

\bibliography{biblio}
\bibliographystyle{amsxport}

\vspace*{.5cm}

\end{document}